\documentclass[12pt,a4paper]{article}
\usepackage{srcltx}
\usepackage[english]{babel}
\usepackage[utf8]{inputenc}
\usepackage[T1]{fontenc}
\usepackage{amssymb,amsfonts,amsmath,amsthm,graphicx}
\usepackage{mathrsfs}
\usepackage{indentfirst}
\usepackage[numbers]{natbib}

\usepackage[pdfauthor={Author's name},%
pdftitle={Document Title},%
pagebackref=true,%
pdftex]{hyperref}

\makeatletter
\newcommand{\subjclass}[2][1991]{%
   \let\@oldtitle\@title%
   \gdef\@title{\@oldtitle\footnotetext{#1 \emph{Mathematics subject classification.} #2}}%
 }

\makeatletter

\theoremstyle{plain}
\newtheorem{teo}{Theorem}[section]
\newtheorem{defn}[teo]{Definition}
\newtheorem{lema}[teo]{Lemma}
\newtheorem{cor}[teo]{Corollary}
\newtheorem{prop}[teo]{Proposition}

\newtheorem{obs}[teo]{Remark}

\DeclareMathOperator{\inv}{Inv}
\DeclareMathOperator{\diam}{diam}
\DeclareMathOperator{\supp}{supp}

\DeclareMathOperator{\re}{Re}

\title{The quaternion over the ring of Colombeau's full generalized
numbers}

\author{Wagner Cortes \and A.~R.~G.~Garcia \and S.~H.~da Silva}

\makeatletter
\newcommand{\Addresses}{{
  \bigskip
  \footnotesize

  Wagner Cortes, \textsc{Instituto de Matemática,
Universidade Federal do Rio Grande do Sul,
Porto Alegre-RS, Brazil,
Av. Bento Gonçalves, 9500,
91509-900}\par\nopagebreak
  \textit{E-mail address}, Wagner Cortes: \texttt{wocortes@gmail.com}

  \medskip

  A.~R.~G.~Garcia (Corresponding author), \textsc{Centro de Ciências Exatas e Naturais,
Universidade Federal Rural do Semi-Árido,
Mossoró-RN, Brazil,
Av. Francisco Mota, 572, 
59.625-900.}\par\nopagebreak
  \textit{E-mail address}, A.~R.~G.~Garcia: \texttt{ronaldogarcia@ufersa.edu.br}

  \medskip

  S.~H.~da Silva, \textsc{Unidade Acadêmica de Matemática,
Universidade Federal de Campina Grande,
Campina Grande-PB, Brazil,
Av. Aprígio Veloso, 785,
58429-970.}\par\nopagebreak
  \textit{E-mail address}, S.~H.~da Silva: \texttt{horacio@mat.ufcg.edu.br}
}}

\begin{document}
\maketitle

\pagestyle{plain}
\pagenumbering{arabic}

\date{}

\begin{abstract}
In this paper, we extend the results obtained by
Cortes-Ferrero-Juriaans (2009) for the
quaternion over the ring Colombeau's simplified generalized numbers,
denoted by $\overline{\mathbb{H}}_s$, to the quaternion over 
the ring of Colombeau's full generalized numbers, denoted by 
$\overline{\mathbb{H}}$. In this paper, we introduce and investigate the
topological algebra of the quaternion over the ring of Colombeau's 
full generalized numbers. This is an important object to study 
if one wants to build the algebraic theory of Colombeau's full generalized
numbers $\overline{\mathbb{K}}$ studied by Aragona-Garcia-Juriaans (2013). 
We study some ring theoretical properties of $\overline{\mathbb{H}}$,  
we classify the dense ideals of $\overline{\mathbb{K}}$ 
in the algebraic sense, and as a consequence, it has a 
maximal ring of quotients which is Von Neumann regular.


\end{abstract}

\section{Introduction}\label{sec-1}
Since its introduction, the theory of Colombeau generalized function
has undergone rapid growth. Fundamental for this theory were the definitions
of Scarpalézos' sharp topologies and the notion of point value by
Kunzinger-Obbergunggenberger. A global theory was developed in \cite{GKS}.

The study of the algebraic aspects of this theory is relatively
recent. This was proposed by J.~Aragona and M.~Obberguggenberger
and started with a paper by Aragona-Juriaans \cite{JO} that extended
its studies for Colombeau's full generalized numbers in \cite{JRJ}.
This, and the developments mentioned above, due to D.~Scarpalézos \cite{SD},
M.~Kunzinger \cite{MK} and M.~Obbergungenberger \cite{OMM} led
Aragona-Fernadez-Juriaans \cite{OJR}, to propose a differential
calculus which in its turn was used to continue
the algebraic aspects of the theory \cite{JO} and \cite{JJO}. 

In \cite{CFJ} focus was on an algebra that may play important role
in the study of the algebraic theory of these algebras. The authors
introduced Colombeau's generalized quaternion algebras, $\overline{\mathbb{H}}_s$,
and studied its topological and some of its algebraic properties. In this paper,
we shall to extend their studies to Colombeau's full generalized
numbers, i.e., we introduced Colombeau's full generalized quaternion
algebras, $\overline{\mathbb{H}}$ (see Definition \ref{colom-6}), and
study its topological 
and some algebraic properties. We study some ring theoretical properties as:
duo (Theorem \ref{colom-16}), exchange (Theorem \ref{colom-18}),
normal (Theorem \ref{colom-20}), Gelfand (Theorem \ref{colom-22}), and
Bezout property (Theorem \ref{colom-23}). We further
classify the dense ideals in the algebraic sense of $\overline{\mathbb{K}}$ and
prove that $\overline{\mathbb{K}}$ and $\overline{\mathbb{H}}$
have a maximal ring of quotients that is Von Newman regular.

In Section \ref{sec-2}, we collect basic definitions, results and
notations to be used throughout the paper and as a rule, whose
proofs were omitted. In Section \ref{sec-3}, we introduced the
topological algebra of the quaternion full generalized numbers, where
we present some results that are extensions of results obtained by
Cortes-Ferrero-Juriaans in \cite{CFJ} and Aragona-Garcia-Juriaans in
\cite{JRJ}. For example, Proposition \ref{colom-13} in Section
\ref{sec-3} extends Proposition \ref{colom-5} that appears 
in Section \ref{sec-2} for the quaternion full generalized numbers. Finally, in
Section \ref{sec-4}, we study some interesting algebraic structure of quaternion
full generalized numbers. Indeed, we study some ring theoretical
properties such as: duo (Theorem \ref{colom-16}), exchange (Theorem \ref{colom-18}),
normal (Theorem \ref{colom-20}), Gelfand (Theorem \ref{colom-22}), and
Bezout property (Theorem \ref{colom-23}) and  related themes.    

The notation used is mostly standard. Some important references for
the theory of Colombeau's full generalized numbers, functions and
their topologies are in \cite{GKS,AB,JFC,MK,OMM,SD}, and more recently, see 
\cite{JRJ,VH} and \cite{CFJ}. See also the set of Notation 1.1 of
\cite{JRJ}, for example, items $a), i), j), l)$ and $m)$ to understand
some notations which appear throughout this paper.

\section{Colombeau's full generalized numbers: a review}\label{sec-2}
In this section, we recall some algebraic theory of Colombeau's full
generalized numbers. We refer the interested reader to
\cite{JRJ,OJR,JO} and \cite{JJO} for notation, more details and proofs of the
results presented in this section. 

The norm of an element $x\in\overline{\mathbb{K}}$ is defined
by $$\Vert x\Vert=D(x,0),$$ where $D$ is the ultra-metric in
$\overline{\mathbb{K}}$ defined in \cite{JRJ} inspired by
Scarpalezos for the Colombeau's simplifies generalized numbers,
$\overline{\mathbb{K}}_s$. Denote by $\inv(\overline{\mathbb{K}})$ 
the unit group of $\overline{\mathbb{K}}$.

Let $$\mathcal{S}_f:=\{A\subset\mathcal{A}_0(\mathbb{K})|\forall~p\in\mathbb{N},~
\exists ~\varphi\in\mathcal{A}_p(\mathbb{K}),~\mbox{such
  that} ~\{\varepsilon|\varphi_\varepsilon\in A\}\in\mathcal{S}\},$$
where $\mathcal{S}:=\{S\subset
]0,1]|0\in\overline{S}\cap\overline{S^c}\}$. 
Here, the bar denotes topological closure. We  denote
$\mathcal{P}_*(\mathcal{S}_f)$ as
the set of all subsets $\mathcal{F}$ of $\mathcal{S}_f$ which are
stable under finite union and such that if $A\in\mathcal{S}_f$, then
$A$ or $A^c$ belongs to $\mathcal{F}$. Also, $g_f(\mathcal{F})=\langle
\mathcal{X}_A:A\in \mathcal{F}\rangle$  
denotes the ideal generated by the characteristic function of $A$, such that
$A\in\mathcal{F}$. We also need to define the set
$$Z(\hat{x}):=\{\varphi\in\mathcal{A}_0(\mathbb{K})|\hat{x}(\varphi)=0\}$$
of the zeros of a representative $\hat{x}$ of $x\in\overline{\mathbb{K}}$. 

Now, we can enunciate the following three results that appears in
\cite{JRJ} and they will be important for our proposal.

\begin{teo}[Fundamental theorem \cite{JRJ}]\label{colom-1}
\begin{enumerate}
\item[$i)$] $x\in
  \inv(\overline{\mathbb{K}})\Leftrightarrow
  Z(\hat{x})\notin \mathcal{S}_f$;
\item[$ii)$] $x\notin \inv(\overline{\mathbb{K}})$ iff there
  exists an idempotent $e\in\overline{\mathbb{K}}$, such that
  $xe=0$. In particular, if
  $x\in\overline{\mathbb{K}}\setminus\{0\}$ and
  $x\notin \inv(\overline{\mathbb{K}})$, then $x$ is a zero
  divisor. Moreover, $\inv(\overline{\mathbb{K}})$ is open and
  dense in $\overline{\mathbb{K}}$.
\end{enumerate}
\end{teo}

\begin{teo}\label{colom-2}
An element $x\in\overline{\mathbb{K}}$ is a unit if and only if
there exists $r>0$ and a map $\tau:\mathcal{A}_0(\mathbb{K})\to
]0,1]$, such that $$|\hat{x}(\varphi_\varepsilon)|\ge\dot{\alpha}_r(\varphi_\varepsilon),~
\forall~0<\varepsilon<\tau(\varphi)=\eta,$$
where $\hat{x}$ is a representative of $x$.
\end{teo}

\begin{teo}\label{colom-3}
Let $\mathfrak{p}\lhd\overline{\mathbb{K}}$ be a prime ideal. Then:
\begin{enumerate}
\item[$a)$] There exists
  $\mathcal{F}_{\mathfrak{p}}\in\mathcal{P}_*(\mathcal{S}_f)$, such
  that $g_f(\mathcal{F}_{\mathfrak{p}})\subset\mathfrak{p}$.
\item[$b)$] $\{\overline{g_f(\mathcal{F})}|\mathcal{F}\in\mathcal{P}_*(\mathcal{S}_f)\}$
  is the set of all maximal ideals of $\overline{\mathbb{K}}$. 
\end{enumerate}
\end{teo}

In \cite{JRJ} it is proved that $g_f(\mathcal{F}_{\mathfrak{p}})$ is
indeed a minimal prime ideal of $\overline{\mathbb{K}}$. In general
$g_f(\mathcal{F}_{\mathfrak{p}})$ is not closed and so
$\overline{\mathbb{K}}$ is not Von Neumann regular.

If $\mathfrak{I}\lhd\overline{\mathbb{K}}$ is a maximal ideal, then
$\overline{\mathbb{K}}$ is algebraically closed in
$\overline{\mathbb{K}}/\mathfrak{I}$ and  it follows that
$\mathcal{B}(\overline{\mathbb{K}})$, the set of idempotents of
$\overline{\mathbb{K}}$, does not depend on $\mathbb{K}$, i.e., 
$\mathcal{B}(\overline{\mathbb{C}})=\mathcal{B}(\overline{\mathbb{R}})$. Moreover,
in (\cite{JRJ},Theorem 4.14) it is proved that
$$\mathcal{B}(\overline{\mathbb{C}})=\mathcal{B}(\overline{\mathbb{R}})
=\{\mathcal{X}_A|A\in\mathcal{S}_f\},$$
where $\mathcal{X}_A$ denotes the characteristic function of the set
$A$, i.e, 
$$\mathcal{X}_A(\varphi)=\left\{\begin{array}{ll}
1, &\mbox{if}~\varphi\in A\\
0, &\mbox{if}~\varphi\in A^c.
\end{array}\right.
$$

For the sake of completeness, we recall the order structure of
$\overline{\mathbb{R}}$  originally defined  in \cite{JFO} and after
in \cite{JRJ}.

\begin{lema}\label{order}
For all $x\in\overline{\mathbb{R}}$ the following  conditions are equivalent:
\begin{enumerate}
\item[$i)$] Every representative $\hat{x}$ of $x$ satisfies the
  condition
\[(*)\left|\begin{array}{l}
\exists ~N\in\mathbb{N},~\mbox{such that}~\forall~b>0
\forall ~\varphi\in\mathcal{A}_N(\mathbb{K})\\
\mbox{there exists}~\eta(b,\varphi)\in ]0,1],~\mbox{such that}\\
\hat{x}(\varphi_\varepsilon)\ge -\varepsilon^b,~\forall~
\varepsilon\in I_\eta=]0,\eta[.
\end{array}\right.
\]
\item[$ii)$] There exists representative $\hat{x}$ of $x$ satisfying
  (*).
\item[$iii)$] There exists a representative $x_*$ of $x$, such that
  $x_*(\varphi)\ge 0, ~\forall~\varphi\in\mathcal{A}_0(\mathbb{K})$.
\item[$iv)$] There exists $N\in\mathbb{N}$ and a representative $x_*$
  of $x$, such that $x_*(\varphi)\ge 0, ~\forall~\varphi\in\mathcal{A}_N(\mathbb{K})$.
\end{enumerate}
\end{lema}

Based in Lemma \ref{order} we have the following definition.

\begin{defn}\label{colom-4}
An element $x\in\overline{\mathbb{R}}$ is said to be non-negative,
quasi-positive or $q$-positive, if it has a representative satisfying one
of the conditions of Lemma \ref{order}. We shall denote this by $x\ge
0$. We shall also say  that $x$ is non-positive, quasi-negative or
$q$-negative if $-x$ is $q$-positive. If $y\in\overline{\mathbb{R}}$
is another element, then we write $x\ge y$ if $x-y$ is $q$-positive and
$x\le y$ if $y-x$ is $q$-positive.
\end{defn}

The Proposition \ref{colom-5} below appears in \cite{JRJ} and we will
extend it for
$\overline{\mathbb{H}}$ in  Proposition \ref{colom-13} in the end
Section \ref{sec-3}. 

\begin{prop}[Convexity of ideals]\label{colom-5}
Let $\mathfrak{J}$ be an ideal of $\overline{\mathbb{K}}$ and
$x,y\in\overline{\mathbb{K}}$. Then:
\begin{enumerate}
\item[$1)$] $x\in\mathfrak{J}$ iff $|x|\in\mathfrak{J}$.
\item[$2)$] If $x\in\mathfrak{J}$ and $|y|\le |x|$, then
  $y\in\mathfrak{J}$.
\item[$3)$] If $\mathbb{K}=\mathbb{R},x\in\mathfrak{J}$ and $0\le y\le
  x$, then $y\in\mathfrak{J}$.  
\end{enumerate}
\end{prop}

Let $r\in\mathbb{R}$. Then $\dot{\alpha}_r\in\overline{\mathbb{R}}$
is the element having $\varphi\in\mathcal{A}_0(\mathbb{K})\mapsto
(i(\varphi))^r\in\mathbb{R}_+$, where
$i(\varphi)=\diam(\supp(\varphi))\ne 0, ~\forall~\varphi\in\mathcal{A}_0(\mathbb{K})$, i.e.,
$\dot{\alpha}_r(\varphi)=(i(\varphi))^r$ as a representative. It has the
property that 
$$\Vert\dot{\alpha}_r\Vert=e^{-r}\qquad\mbox{and}\qquad\Vert\dot{\alpha}_rx\Vert=\Vert\dot{\alpha}_r\Vert\Vert
x\Vert,$$ for any $x\in\overline{\mathbb{K}}$. Thus, we have that an
element $x\in\overline{\mathbb{K}}$ is a unit iff there exists
$r\in\mathbb{R}$, such that $|x|\ge\dot{\alpha}_r$ (see Theorem \ref{colom-2}).

In \cite{JRJ}, it is also proved that for
$0<x\in\overline{\mathbb{K}}$, there exists
$y\in\overline{\mathbb{K}}$, such that $x=y^2$. In Sections
\ref{sec-3} and \ref{sec-4}, we will use freely some of the 
results in this section.

\section{The topological algebra of the quaternion over the ring of Colombeau's
full generalized numbers}\label{sec-3}

Here, $(\mathbb{H},|\cdot|)$ will denote the classical ring of real
quaternion with usual metric and  basis $\{1,i,j,k\}$ and
$\mathbb{K=\mathbb{R}}$ unless otherwise stated. Moreover, if $S$ is a ring,
then $\overline{\mathbb{H}}(S)$ denotes the quaternion algebra
over $S$. We begin with the following definition  about moderate and null functions.

\begin{defn}\label{colom-6}
\begin{enumerate}
\item[$i)$] A function $\hat{x}:\mathcal{A}_0(\mathbb{K})\to \mathbb{H}$ is
moderate if there exists $p\in\mathbb{N}$, such that for all
$\varphi\in\mathcal{A}_p(\mathbb{K})$, there exists $C=C_\varphi>0$
and $\eta=\eta_\varphi>0$, such
that $$|\hat{x}(\varphi_\varepsilon)|\le C\varepsilon^{-p},~\forall~0<\varepsilon<\eta.$$ 
\item[$ii)$] Let
$\mathcal{E}_M(\mathbb{H})=\{\hat{x}:\mathcal{A}_0(\mathbb{K})\to\mathbb{H}|\hat{x}~
\mbox{is moderate}\}$ and
$\mathcal{N}(\mathbb{H})=\{\hat{x}\in\mathcal{E}_M(\mathbb{H})|\hat{x}~\mbox{is
  null}\}$. We say that $\hat{x}\in\mathcal{E}_M(\mathbb{H})$ is
null if there exists $p\in\mathbb{N}$ and $\gamma\in\Gamma$, such that
for all $\varphi\in\mathcal{A}_q(\mathbb{K}),~q\ge p$ there exists
$C=C_\varphi>0$ and $\eta=\eta_\varphi>0$, such
that $$|\hat{x}(\varphi_\varepsilon)|\le
C\varepsilon^{\gamma(q)-p},~\forall ~0<\varepsilon<\eta,$$ where  
$\Gamma:=\{\gamma:\mathbb{N}\to\mathbb{R}^+|\gamma(n)<\gamma(n+1),~
\forall~n\in\mathbb{N}~\mbox{and}~\lim\limits_{n\to\infty}\gamma(n)=\infty\}$
is the set of the strict increasing sequences diverging to infinity
when $n\to\infty$.
\end{enumerate}
\end{defn}

\begin{lema}\label{colom-7}
We have the following isomorphism
\[\mathcal{E}_M(\mathbb{H})/\mathcal{N}(\mathbb{H})\cong
\mathbb{H}(\mathcal{E}_M(\mathbb{R}))/\mathbb{H}(\mathcal{N}(\mathbb{R}))
=\mathbb{H}(\overline{\mathbb{R}}).\] 
\end{lema}

\begin{proof}
To prove this, we need to show that there exists the isomorphisms:
\begin{enumerate}
\item[$i)$]
  $\mathcal{E}_M(\mathbb{H})\cong\mathbb{H}(\mathcal{E}_M(\mathbb{R}))$;
\item[$ii)$] $\mathcal{N}(\mathbb{H})\cong\mathbb{H}(\mathcal{N}(\mathbb{R}))$.
\end{enumerate}
In $i)$ we need to show that
$\hat{x}(\varphi)=\hat{x}_0(\varphi)1+\hat{x}_1(\varphi)i+
\hat{x}_2(\varphi)j+\hat{x}_3(\varphi)k\in\mathcal{E}_M(\mathbb{H})$
if and only if $\hat{x}_n(\varphi)\in\mathcal{E}_M(\mathbb{R}),~\forall
~n=0,1,2,3$, i.e.,
$\hat{x}(\varphi)\in\mathbb{H}(\mathcal{E}_M(\mathbb{R}))$. Indeed,
$\hat{x}(\varphi)\in\mathcal{E}_M(\mathbb{H})\Leftrightarrow
\exists~p\in\mathbb{N}$, such that for all
$\varphi\in\mathcal{A}_p(\mathbb{K})$, there exists $C=C_\varphi>0$ and
$\eta=\eta_\varphi>0$, such that $|\hat{x}(\varphi_\varepsilon)|\le
C\varepsilon^{-p},~\forall~0<\varepsilon<\eta
\Leftrightarrow
|\hat{x}_0(\varphi_\varepsilon)1+\hat{x}_1(\varphi_\varepsilon)i+
\hat{x}_2(\varphi_\varepsilon)j+\hat{x}_3(\varphi_\varepsilon)k|\le
C\varepsilon^{-p},~\forall~0<\varepsilon<\eta\Leftrightarrow
|\hat{x}_n(\varphi_\varepsilon)|\le
C\varepsilon^{-p},~\forall~n=0,1,2,3~\mbox{and}~
\forall~0<\varepsilon<\eta\Leftrightarrow\hat{x}_n(\varphi)\in\mathcal{E}_M(\mathbb{R}),
~\forall~n=0,1,2,3\Leftrightarrow\hat{x}(\varphi)\in\mathbb{H}(\mathcal{E}_M(\mathbb{R}))$.

By similar methods of  $i)$ we have that 
$\mathcal{N}(\mathbb{H})\cong\mathbb{H}(\mathcal{N}(\mathbb{R}))$. Hence,
we conclude
that $$\mathcal{E}_M(\mathbb{H})/\mathcal{N}(\mathbb{H})\cong 
\mathbb{H}(\mathcal{E}_M(\mathbb{R}))/\mathbb{H}(\mathcal{N}(\mathbb{R}))=\mathbb{H}(\overline{\mathbb{R}}),$$
and it follows the assertion.
\end{proof}

We denote $\mathbb{H}(\overline{\mathbb{R}})$ by
$\overline{\mathbb{H}}$. Let
$x=x_0+x_1i+x_2j+x_3k\in\overline{\mathbb{H}}$ and its conjugate
$\bar{x}:=x_0-(x_1i+x_2i+x_3k)$, we have that
$n(x):=\sqrt{x_0^2+x_1^2+x_2^2+x_3^2}=\sqrt{x\bar{x}}$ is the norm
of $x\in\overline{\mathbb{H}}$. This defines  a function of
$\overline{\mathbb{H}}$ in $\overline{\mathbb{R}}_+$ and we get that each 
$x\in\overline{\mathbb{H}}$ associates to $n(x)=\sqrt{x\bar{x}}\in\overline{\mathbb{R}}_+$. 

\begin{lema}\label{unit}
Let $x\in\overline{\mathbb{H}}$. Then 
\begin{enumerate}
\item[$a)$] $x\in \inv(\overline{\mathbb{H}})$ if and only if
  $n(x)\in \inv(\overline{\mathbb{R}})$ and in this case, we have
  that $x^{-1}=(n(x))^{-1}\bar{x}$.
\item[$b)$] An element $x\in\overline{\mathbb{H}}$ is a zero divisor
  if and only if $n(x)$ is a zero divisor. In particular, we have that
  an element of $\overline{\mathbb{H}}$ is either a unit or a zero divisor.
\item[$c)$] Let $x=x_0+x_1i+x_2j+x_3k$. If one of the $x_n$,
  $n=0,1,2,3$ is a unit, then $x$ is a unit.   
  \end{enumerate}
\end{lema}

\begin{proof}
The proof is the same as \cite{CFJ}.
\end{proof}

\begin{defn}\label{colom-8}
An element $x\in\overline{\mathbb{H}}$ is associated to zero, a
property which will be denoted by $x\approx 0$, if for some (or
equivalently, for each) representative $(x(\varphi))_\varphi$ of $x$
we have \[\exists~p\in\mathbb{N}~\mbox{such
  that}~\lim\limits_{\varepsilon\downarrow
  0}x(\varphi_\varepsilon)=0~\forall
~\varphi\in\mathcal{A}_p(\mathbb{K}).\] Two elements
$x,y\in\overline{\mathbb{H}}$ are associated if $x-y\approx 0$, a
property which will be denoted by $x\approx y$. If there exists some
$a\in\mathbb{H}$ with $x\approx a$, then $a$ is called associated
quaternion or shadow of $x$.
\end{defn}

\begin{lema}\label{colom-9}
Let $x,y\in\overline{\mathbb{H}}$ and $a\in\overline{\mathbb{H}}$. Then
\begin{enumerate}
\item[$a)$] $x\approx 0\Leftrightarrow x_n\approx 0, ~\forall~n=0,1,2,3$;
\item[$b)$] $x\approx y\Leftrightarrow x_n\approx y_n,~\forall~n=0,1,2,3$;
\item[$c)$] $x\approx a,~x=x_0+x_1i+x_2j+x_3k\in\overline{\mathbb{H}}$ if and only if there exists
  $a_0,a_1,a_2,a_3\in\mathbb{R}$ such $x_n\approx
  a_n,~\forall~n=0,1,2,3$, i.e, $a_n,~n=0,1,2,3$ are associated
  numbers or shadows of $x_n\in\overline{\mathbb{R}},~n=0,1,2,3$,
  respectively (see Definition 1.3 in \cite{JRJ}).
\end{enumerate}
\end{lema}

\begin{proof}
$a)$: Let
$x=x_0+x_1i+x_2j+x_3k\in\overline{\mathbb{H}}$. Then $x\approx
0\Leftrightarrow \exists~ p\in\mathbb{N}, ~\mbox{such that}
\lim\limits_{\varepsilon\downarrow
  0}x(\varphi_\varepsilon)=0,~\forall~\varphi\in\mathcal{A}_p(\mathbb{K})\Leftrightarrow
 \lim\limits_{\varepsilon\downarrow 0}(x_0(\varphi_\varepsilon)+
x_1(\varphi_\varepsilon)i+x_2(\varphi_\varepsilon)j+x_3(\varphi_\varepsilon)k)=0,
~\forall~\varphi\in\mathcal{A}_p(\mathbb{K})\Leftrightarrow \lim\limits_{\varepsilon\downarrow
  0}x_0(\varphi_\varepsilon)+ \lim\limits_{\varepsilon\downarrow
  0}x_1(\varphi_\varepsilon)i+\lim\limits_{\varepsilon\downarrow
  0}x_2(\varphi_\varepsilon)j+\lim\limits_{\varepsilon\downarrow
  0}x_3(\varphi_\varepsilon)k=0,~\forall~\varphi\in\mathcal{A}_p(\mathbb{K})\Leftrightarrow
\lim\limits_{\varepsilon\downarrow
  0}x_n(\varphi_\varepsilon)=0,~n=0,1,2,3,~\forall~\varphi\in\mathcal{A}_p(\mathbb{K})\Leftrightarrow
x_n\approx 0,~\forall~n=0,1,2,3$ (see Definition 1.3 in \cite{JRJ}).

 $b)$: Let $x=x_0+x_1i+x_2j+x_3k$ and
$y=y_0+y_1i+y_2j+y_3k$ be elements of $\overline{\mathbb{H}}$. Then
$x\approx y\Leftrightarrow (x_0+x_1i+x_2j+x_3k)\approx
(y_0+y_1i+y_2j+y_3k)\Leftrightarrow
(x_0-y_0)+(x_1-y_1)i+(x_2-y_2)j+(x_3-y_3)k\approx 0\Leftrightarrow
(x_n-y_n)\approx 0,~\forall~n=0,1,2,3\Leftrightarrow x_n\approx y_n,
~\forall ~n=0,1,2,3$ (see Definition 1.3 in \cite{JRJ}).

 $c)$: The proof is standard.
\end{proof}

For an element $x\in\overline{\mathbb{H}}$,
let \[A(x):=\{r\in\mathbb{R}|\dot{\alpha}_{-r}x\approx 0\},\] where
$\dot{\alpha}_{-r}(\varphi):=(i(\varphi))^{-r}\in\overline{\mathbb{R}}$
with inverse $\dot{\alpha}_r(\varphi)=(i(\varphi))^r$ (see
\cite{JRJ}), and \[V(x):=\sup(A(x))\] its valuation. 
In the next lemma, we  characterize the elements of $A(x)$.

\begin{lema}\label{norm-1}
Let $x\in\overline{\mathbb{H}}$. Then $r\in A(x)$ if and only if
there exists $p\in\mathbb{N}$, such
that \[\lim\limits_{\varepsilon\downarrow
  0}\varepsilon^{-r}x_n(\varphi_\varepsilon)=0,~\forall
~\varphi\in\mathcal{A}_p(\mathbb{K}),~n=0,1,2,3.\]
This means that $r\in A(x)\Leftrightarrow r\in A(x_n),~n=0,1,2,3$
(see Definition 1.4 in \cite{JRJ}).
\end{lema}

\begin{proof}
Let $x=x_0+x_1i+x_2j+x_3k\in\overline{\mathbb{H}}$. Then $r\in
A(x)\Leftrightarrow \dot{\alpha}_{-r}x\approx 0\Leftrightarrow
\dot{\alpha}_{-r}(x_0+x_1i+x_2j+x_3k)\approx 0\Leftrightarrow
\dot{\alpha}_{-r}x_0+\dot{\alpha}_{-r}x_1i+\dot{\alpha}_{-r}x_2j+\dot{\alpha}_{-r}x_3k\approx
0\Leftrightarrow \dot{\alpha}_{-r}x_n\approx
0,~n=0,1,2,3\Leftrightarrow \exists ~p\in\mathbb{N}$, such that
$\lim\limits_{\varepsilon\downarrow
  0}\varepsilon^{-r}(i(\varphi))^{-r}x_n(\varphi_\varepsilon)=0,~n=0,1,2,3,~\forall~\varphi\in\mathcal{A}_p(\mathbb{K})\Leftrightarrow
(i(\varphi))^{-r}\lim\limits_{\varepsilon\downarrow
  0}\varepsilon^{-r}x_n(\varphi_\varepsilon)=0,~n=0,1,2,3,~\forall~\varphi\in\mathcal{A}_p(\mathbb{K})\Leftrightarrow
\lim\limits_{\varepsilon\downarrow
  0}\varepsilon^{-r}x_n(\varphi_\varepsilon)=0,~n=0,1,2,3,~\forall~\varphi\in\mathcal{A}_p(\mathbb{K})$,
once that $(i(\varphi))^{-r}\ne 0,
~\forall~\varphi\in\mathcal{A}_p(\mathbb{K})$. From Definition 1.4 of
\cite{JRJ} it is equivalent to $r\in A(x)\Leftrightarrow r\in A(x_n),~\forall~n=0,1,2,3$.
\end{proof}

It is easily seen that either $A(x)=\mathbb{R}$ or there exists
$r\in\mathbb{R}$, such that either $A(x)=]-\infty,r[$ or
$A(x)\subseteq ]-\infty,r]$. We define $\Vert x\Vert=e^{-V(x)}$ and
$d(x,y):=\Vert x-y\Vert$. Note  that all the definitions above
make sense and that $d$ defines an ultra metric on
$\overline{\mathbb{H}}$. We denote $d_{\pi}$ as the product metric on
$\overline{\mathbb{H}}$ induced by the topology of
$\overline{\mathbb{K}}$ defined in \cite{JRJ}.

\begin{teo}\label{norm-2}
$(\overline{\mathbb{H}},d)$ and $(\overline{\mathbb{H}},d_\pi)$
are homeomorphic as topological spaces, i.e., $d$ and $d_\pi$ defines
the same topology on $\overline{\mathbb{H}}$.
\end{teo} 

\begin{proof}
Let $x=(x_0+x_1i+x_2j+x_3k)\in(\overline{\mathbb{H}},d)$ and $r>0$ and
let $$B_r(x)=\{y\in\overline{\mathbb{H}}|d(x,y)<r\}.$$ If
$y=(y_0+y_1i+y_2j+y_3k)\in B_r(x)$, then 
\begin{eqnarray*}
d(x,y)<r&\Leftrightarrow&\Vert x-y\Vert<r\\
&\Leftrightarrow&e^{-V(x-y)}<r\\
&\Leftrightarrow&e^{V(x-y)}>\frac{1}{r}.
\end{eqnarray*}
Equivalently, we have that
\[V(x-y)>\ln\left(\frac{1}{r}\right)\Leftrightarrow
\ln\left(\frac{1}{r}\right)\in A(x-y).\]
Now, by the Lemma \ref{norm-1}, we obtain,
for all $n=0,1,2,3$, that
\begin{eqnarray*}
\ln\left(\frac{1}{r}\right)\in A(x_n-y_n)
&\Leftrightarrow&V(x_n-y_n)>\ln\left(\frac{1}{r}\right),~n=0,1,2,3\\
&\Leftrightarrow&e^{V(x_n-y_n)}>\frac{1}{r},~n=0,1,2,3\\
&\Leftrightarrow&e^{-V(x_n-y_n)}<r,~n=0,1,2,3\\
&\Leftrightarrow&\Vert x_n-y_n\Vert<r,~n=0,1,2,3\\
&\Leftrightarrow&d_{\pi}(x,y)<r.
\end{eqnarray*}
Therefore, it follows the assertion.
\end{proof}

\begin{cor}\label{colom-10}
$(\overline{\mathbb{H}},d)$ is a complete metric algebra.
\end{cor}

\begin{proof}
The proof follows from Theorem \ref{norm-2} and the Proposition 1.7 of
\cite{JRJ} which assures that $\overline{\mathbb{K}}$ is a complete
topological ring.
\end{proof}

The following result shows that the unit group of
$\overline{\mathbb{H}}$ is very ``big''.

\begin{teo}\label{colom-11}
The unit group $\inv(\overline{\mathbb{H}})$ is an open and dense subset
of $\overline{\mathbb{H}}$.
\end{teo}

\begin{proof}
Let $x=a_0+a_1i+a_2j+a_3k\in \inv(\overline{\mathbb{H}})$. Then
$n(x)=\sqrt{x\bar{x}}\in \inv(\overline{\mathbb{R}})$. Thus, by the
Fundamental Theorem, there exists $r>0$, such that
$B_r(n(x))\subseteq \inv(\overline{\mathbb{R}})$. Now, let us suppose
that $ \inv(\overline{\mathbb{H}})$ is not an open set. Then, still
by the same theorem, we have that for all $n\in\mathbb{N}^*$, there
exists a sequence $x_n\in B_{\frac{1}{n}}(x)$, such that
$n(x_n)\notin \inv(\overline{\mathbb{R}})$ and
$x_n\underset{n\to\infty}{\to} x$. Hence, by Theorem \ref{norm-2}, we
have that $(x_n)_i\underset{n\to\infty}{\to} x_i,~i=0,1,2,3$ and it 
follows
that \[n(x_n)=\sqrt{\sum_{i=0}^3[(x_n)_i]^2}\underset{n\to\infty}{\to}\sqrt{\sum_{i=0}^3x_i^2}=n(x),\]
which contradicts the fact that $n(x)\in \inv(\overline{\mathbb{R}})$.

We now prove the density. Suppose that there exists $r\in\mathbb{R}$
and $z\in\overline{\mathbb{H}}$, such that
$B_r(z)\cap \inv(\overline{\mathbb{H}})=\emptyset$. If
$x=n(z)\in\overline{\mathbb{R}}$, and since the norm $n$ is
obviously a continuous function, then  we have that there exists an open
ball $B_s(x)\subset\overline{\mathbb{R}}$, such that
$n^{-1}(B_s(y))\subset B_r(z)$ which  according to Lemma \ref{unit},
contradicts the density of $ \inv(\overline{\mathbb{R}})$ in
$\overline{\mathbb{R}}$ (see \cite{JRJ}).   
\end{proof}

\begin{cor}\label{colom-12}
Let $\mathfrak{M}$ be an ideal of $\overline{\mathbb{H}}$. If
$\mathfrak{M}$ is maximal ideal of $\overline{\mathbb{H}}$, 
then $\mathfrak{M}=\overline{\mathfrak{M}}$.
\end{cor}

According to \cite{SSJ}, each element $y$ in the real quaternion number
$\mathbb{H}$ has the polar decomposition $y=n(y)e^{ar}$, where $r$ is
the norm of the imaginary part of $y$ and $a$ is the angle
$-\pi\leq a\leq \pi$ where we say that $a$ is the argument of $y$ 
and we denote it by $\arg(y)$. Remember that given a $y\in\mathbb{H}$,
we have that $\re(y)$ is the real part of $y$ and $\Im m(y)$ is
imaginary part of $y$.

Now, for an element $y\in\overline{\mathbb{H}}$ we define the map
$\theta_{\hat{y}}:\mathcal{A}_0(\mathbb{K})\to\mathbb{H}$ by 
$$\theta_{\hat{y}}(\varphi)=\exp(r.\arg(\hat{y}(\varphi))$$ and its
inverse the map $\theta_{\hat{y}}^{-1}:\mathcal{A}_0(\mathbb{K})\to\mathbb{H}$ by
$$\theta^{-1}_{\hat{y}}(\varphi)=\exp(-r\arg(\hat{y}(\varphi)),$$ 
where $-\pi\leq \arg(\hat{y}(\varphi))\leq \pi,
~r=n(\Im m(\hat{y}(\varphi)))$ and $\hat{y}:A_0(\mathbb{K})\rightarrow
\mathbb{H}$ is a representative of $y$. Thus, for each 
$\varphi\in \mathcal{A}_0(\mathbb{K})$, we have that
$\theta_{\hat{y}}(\varphi).n(\hat{y}(\varphi))=\hat{y}(\varphi)$. Therefore,
\begin{equation}\label{polar}
\theta_{\hat{y}}.n(\hat{y})=\hat{y},
\end{equation}
and we have that $\forall ~y\in\overline{\mathbb{H}}$, there exists a (and
therefore all) representative(s) $\hat{y}$ of
$y$ such that holds the equation (\ref{polar}).   
Since $\theta_{\hat{y}}$ and $\theta^{-1}_{\hat{y}}$ are inverses of each other and  
we denote their classes by $\Theta_{\hat{y}}$ and
$\Theta^{-1}_{\hat{y}}$, respectively. 

In Proposition \ref{colom-5} we have the convexity of ideals in
$\overline{\mathbb{R}}$ and in the next result we have the convexity 
of ideals in $\overline{\mathbb{H}}$ 
which completes (\cite{CFJ}, Lemma 3.7).

\begin{prop} \label{colom-13}
Let $\mathfrak{J}$ be an ideal of $\overline{\mathbb{H}}$. 
The following statements hold.
\begin{enumerate}
\item[$a)$] $x\in\mathfrak{J}$ if and only if $n(x)\in \mathfrak{J}$.
\item[$b)$] If $x\in\mathfrak{J}$ and $n(y)\leq n(x)$, then $y\in \mathfrak{J}$.
\end{enumerate}
\end{prop}

\begin{proof} $a)$: Suppose that $x\in \mathfrak{J}$. Since,
  $x=n(x)\Theta_{\hat{x}}$ then we have that   
$n(x)=x\Theta^{-1}_{\hat{x}}\in \mathfrak{J}$.
Conversely, suppose that  $n(x)\in \mathfrak{J}$ and we have that 
$x=n(x)\Theta^{-1}_{\hat{x}}\in \mathfrak{J}$. So, $x\in \mathfrak{J}$

$b)$: It follows from the similar methods of the proof 
(\cite{JRJ}, Proposition 4.6).
\end{proof}

\section{The algebraic structure of $\overline{\mathbb{H}}$}\label{sec-4}
In this section, unless otherwise stated, $\mathbb{K}=\mathbb{R}$. 
We start proving that Boolen algebra of $\overline{\mathbb{H}}$ 
equals that of $\overline{\mathbb{R}}$.

\begin{teo}\label{colom-14}
$\mathcal{B}(\overline{\mathbb{H}})=\mathcal{B}(\overline{\mathbb{R}})$.
\end{teo}

\begin{proof}
To prove this, it is necessary to prove the following two inclusions:
\begin{enumerate}
\item[$i)$]
  $\mathcal{B}(\overline{\mathbb{H}})\subseteq\mathcal{B}(\overline{\mathbb{R}})$;
\item[$ii)$] $\mathcal{B}(\overline{\mathbb{R}})\subseteq\mathcal{B}(\overline{\mathbb{H}})$.
\end{enumerate}
Let us prove $i)$: Let $e\in\mathcal{B}(\overline{\mathbb{H}})$ be a
non-trivial idempotent. Then $n(e),n(1-e)\in\overline{\mathbb{R}}$
are idempotents, i.e.,
$n(e),n(1-e)\in\mathcal{B}(\overline{\mathbb{R}})$. By the Theorem
4.14 of \cite{JRJ}, there exists $A,B\in\mathcal{S}_f$, such that
$n(e)=\mathcal{X}_A$ and $n(1-e)=\mathcal{X}_B$. Thus, we have that
that $$n(e)\mathcal{X}_{A^c}=\mathcal{X}_A\mathcal{X}_{A^c}=0,$$ but
$n(e)\mathcal{X}_{A^c}=n(e\mathcal{X}_{A^c})$ which implies that 
$n(e\mathcal{X}_{A^c})=0$ Hence,   $e=e\mathcal{X}_A$. 
On the other hand, since  $e(1-e)=0$, we have that
$\mathcal{X}_A\mathcal{X}_B=0$. By the fact that   $e\mathcal{X}_A=e$ and
$(1-e)\mathcal{X}_B=1-e$, we have  that
$e\mathcal{X}_A+(1-e)\mathcal{X}_B=1$. Consequently, 
$e\mathcal{X}_A=\mathcal{X}_A$ and we get that  $e=\mathcal{X}_A$. Hence,
$e\in\mathcal{B}(\overline{\mathbb{R}})$ and it follows $i)$. 
The affirmation in $ii)$ is clearly true. From $i)$ and $ii)$  the result follows.
\end{proof}

The following remark will be useful in this paper.

\begin{obs}\label{ideal-2}
Note that
$\overline{\mathbb{H}}/\overline{\mathbb{H}}(g_f(\mathcal{F}))$ is
isomorphic to $\overline{\mathbb{H}}\left(\overline{\mathbb{R}}/g_f(\mathcal{F})\right)$.
\end{obs}

In the next result, we shall use the Fundametal Theorem of
$\overline{\mathbb{K}}$, i.e, Theorem \ref{colom-1}, to
give a complete description of the maximal ideals of
$\overline{\mathbb{H}}$.

\begin{teo}\label{ideal-1}
Let $\mathfrak{M}$ be a maximal ideal of
$\overline{\mathbb{H}}$. Then there exists
$\mathcal{F}\in\mathcal{P}_*(S_f)$, such that
$\mathfrak{M}=\overline{\mathbb{H}}\left(\overline{g_f(\mathcal{F})}\right)$.
\end{teo}

\begin{proof}
We clearly have that $\mathfrak{M}\cap\overline{\mathbb{R}}$ is a
prime ideal. Thus, there exists a unique
$\mathcal{F}\in\mathcal{P}_*(\mathcal{S}_f)$, such
that $$g_f(\mathcal{F})\subseteq
\mathfrak{M}\cap\overline{\mathbb{R}}\subseteq \mathfrak{M}.$$  Hence,
we have
that $$\overline{g_f(\mathcal{F})}\subseteq\overline{\mathfrak{M}}=\mathfrak{M},$$
which implies that
$\overline{g_f(\mathcal{F})}=\mathfrak{M}\cap\overline{\mathbb{R}}$. Consequently,
$\mathfrak{M}\cap\overline{\mathbb{R}}$ is a maximal
ideal. Since,  $$\overline{\mathbb{H}}(\overline{g_f(\mathcal{F})})=\overline{g_f(\mathcal{F})}1+
\overline{g_f(\mathcal{F})}i+\overline{g_f(\mathcal{F})}j+\overline{g_f(\mathcal{F})}k,$$
 then  $$\overline{\mathbb{H}}/\overline{\mathbb{H}}(\overline{g_f(\mathcal{F})})\cong
\overline{\mathbb{H}}(\overline{\mathbb{R}}/\overline{g_f(\mathcal{F})})$$
and by the fact that  $\mathbb{R}$ is algebraically closed in
$\overline{\mathbb{R}}/\overline{g_f(\mathcal{F})}$ and $\overline{\mathbb{R}}$ is simple, 
 we have  that  $\mathfrak{M}=\overline{\mathbb{H}}(\overline{g_f(\mathcal{F})})$.  
\end{proof}

Let $\mathfrak{I}$ be an ideal of $\overline{\mathbb{H}}$  and  denoted
by $n(\mathfrak{I})$, the ideal of $\overline{\mathbb{R}}$ generated
by the set $\{n(x)|x\in\mathfrak{I}\}$.

The next definition is well known, see \cite{CR}.

\begin{defn}\label{colom-15} A ring $S$  is said to be right(left) duo if all 
right(left) ideals of $S$  are two-sided ideals. Moreover, $S$  
is duo if it is right and left duo. 
\end{defn}  

Now, we prove an interesting result for the generalized quaternion rings.

\begin{teo}\label{colom-16}
$\overline{\mathbb{H}}$ is duo.
\end{teo}

\begin{proof}
Let $\mathfrak{I}$ be a right ideal of $\overline{\mathbb{H}}$ and  we claim 
that $\langle \mathfrak{I}\cap \overline{R} \rangle=\{ x\in
\overline{\mathbb{H}}:n(x)\in \mathfrak{I}\cap \overline{\mathbb{R}}\}=\mathfrak{I}$. 
In fact, it is not  difficult to show that $\langle \mathfrak{I}\cap
\overline{\mathbb{R}} \rangle$ is an ideal of $\overline{\mathbb{H}}$.  
Let $y\in \mathfrak{I}$. Then we have that $n(y)^2=y.\overline{y}\in \mathfrak{I}\cap
\overline{\mathbb{R}}$. Since $n(y)\leq n(y)^2\in \mathfrak{I}\cap
\overline{\mathbb{R}}$, then by Proposition \ref{colom-13}-$b)$, we have 
that $n(y)\in \mathfrak{I}\cap \overline{\mathbb{R}}$. Thus, 
$y\in \langle \mathfrak{I}\cap \overline{\mathbb{R}}\rangle$ 
and we have that $\mathfrak{I}\subseteq \langle \mathfrak{I}\cap \overline{\mathbb{R}}\rangle$.
On the other hand,  for each $a\in \langle \mathfrak{I}\cap
\overline{\mathbb{R}}\rangle$ we have that $n(a)\in \mathfrak{I}\cap
\overline{\mathbb{R}}\subseteq \mathfrak{I}$. Since, $n(a)=a\Theta_{\hat{a}},$ 
then $a=n(a).\Theta^{-1}_{\hat{a}}\in \mathfrak{I}$. So, $\mathfrak{I}=\langle \mathfrak{I}\cap
\overline{\mathbb{R}}\rangle$ 
and we obtain that $\mathfrak{I}$ is an ideal of $\overline{\mathbb{H}}$. 
Similarly, we prove that the left ideals are ideals. 
Therefore,  $\overline{\mathbb{H}}$ is duo.
\end{proof}

Let $S$ be a ring, and denote by $\gamma(S)$, its Brown McCoy radical, i.e.,
$\gamma(S)$ is the intersection of all ideals $\mathfrak{M}$ of $S$, such
that $S/\mathfrak{M}$ is a simple and unitary. Note that if $S$ is either 
commutative or duo, then $\gamma(S)$ coincide with the Jacobson
radical. In the next result, we completely characterize the Jacobson
radical of $\overline{\mathbb{H}}$, since by Theorem \ref{colom-16}, we have that 
$\overline{\mathbb{H}}$ is duo.

\begin{lema}\label{ideal-3}
$J(\overline{\mathbb{H}})\cap\overline{\mathbb{R}}=J(\overline{\mathbb{R}})$. In
particular, $J(\overline{\mathbb{H}})=(0)$.
\end{lema}  

\begin{proof}
To prove this, we need to prove the both inclusions below:
\begin{enumerate}
\item[$i)$]
  $J(\overline{\mathbb{H}})\cap\overline{\mathbb{R}}\subseteq J(\overline{\mathbb{R}})$;
\item[$ii)$] $J(\overline{\mathbb{R}})\subseteq J(\overline{\mathbb{H}})\cap\overline{\mathbb{R}}$. 
\end{enumerate}

Let us prove $i)$: Let $\mathfrak{M}$ be a maximal ideal of
$\overline{\mathbb{R}}$. Then by the Theorem 3.21-(1) of \cite{JRJ},
we have that $\mathfrak{M}=\overline{g_f(\mathcal{F})}$ for some
$\mathcal{F}\in\mathcal{P}_*(\mathcal{S}_f)$ and we get that 
$\overline{\mathbb{H}}(\mathfrak{M})=\overline{\mathbb{H}}(\overline{g_f(\mathcal{F})})$ is
a maximal ideal of $\overline{\mathbb{H}}$, which implies that
$J(\overline{\mathbb{R}})\supseteq
J(\overline{\mathbb{H}})\cap\overline{\mathbb{R}}$. 

Now we will prove $ii)$: Let $\mathfrak{M}$ be a maximal ideal of
$\overline{\mathbb{H}}$. Then by the Theorem \ref{ideal-1}, we have
that $\mathfrak{M}=\overline{\mathbb{H}}(\overline{g_f(\mathcal{F})})$ for some
$\mathcal{F}\in\mathcal{P}_*(\mathcal{S}_f)$, but in this case, we
have that
$\overline{g_f(\mathcal{F})})=\mathfrak{M}\cap\overline{\mathbb{R}}$
is a maximal ideal of $\overline{\mathbb{R}}$ and thus
$J(\overline{\mathbb{R}})\subseteq
J(\overline{\mathbb{H}})\cap\overline{\mathbb{R}}$. 

From $i)$ and $ii)$, we have that
$J(\overline{\mathbb{H}})\cap\overline{\mathbb{R}}=J(\overline{\mathbb{R}})$. 

By (\cite{JRJ}, Theorem 3.20) the Jacobson radical of
$\overline{\mathbb{R}}$ is zero and it follows that 
$J(\overline{\mathbb{H}})\cap \overline{\mathbb{R}}=0$.  
So, $J(\overline{\mathbb{H}})=0$. 
\end{proof}

Using similar ideals of (\cite{JRJ}, Theorem 4.17) and (\cite{CFJ},
Proposition 4.5), we have the following Proposition.

\begin{prop}\label{ideal-4}
The ring $\overline{\mathbb{H}}$ is not Von Neumann regular.
\end{prop}

According to \cite{ARA}, a noncommutative ring $S$ is exchange if for
any $a\in S$ there exists an idempotent $e\in S$ 
such that $a+e$ is invertible.

Next, we want to show that $\overline{\mathbb{H}}$ is an exchange ring 
and, for this, we need to prove that $\overline{\mathbb{R}}$ is an
exchange ring as the next proposition shows.

\begin{prop}\label{colom-17}
$\overline{\mathbb{K}}$ is an exchange ring. In particular, if
$\mathbb{K}=\mathbb{R}$, then $\overline{\mathbb{R}}$ is an exchange ring.
\end{prop}

\begin{proof} Let $x\in\overline{\mathbb{K}}, ~\hat{x}$ its
  representative, \[T=\{\varphi\in
  \mathcal{A}_0(\mathbb{K}):\exists ~p\in\mathbb{N} ~\mbox{such
  that} ~|\hat{x}(\varphi_{\varepsilon})|\le
1/2,~\forall~\varphi\in\mathcal{A}_p(\mathbb{K}),
~\forall~\varepsilon\in I_{\eta}\},\] where $I_\eta\subset (0,1)$, and 
$e_T=\mathcal{X}_T$. We claim that $x+e_T$ is an invertible
element. i.e, $|(\hat{x}+\hat{e}_t)(\varphi_\varepsilon)|\ge\frac{1}{2}$
for $r=\log_\varepsilon(\frac{1}{2i(\varphi)})$ in
$\dot{\alpha}_r(\varphi_\varepsilon)=\frac{1}{2}$ as in Theorem \ref{colom-2}. 
In fact, if $\varphi\in T$, then there exists $p\in\mathbb{N} ~\mbox{such
  that} ~|\hat{x}(\varphi_{\varepsilon})|\le
1/2,~\forall~\varphi\in\mathcal{A}_p(\mathbb{K}),
~\forall~\varepsilon\in I_{\eta}$. Hence,
\begin{eqnarray}\label{exchan-1}  
|(\hat{x}+\hat{e}_T)(\varphi_{\varepsilon})|&=&|\hat{x}(\varphi_{\varepsilon})+
\hat{e}_T(\varphi_{\varepsilon})|\nonumber\\
&=&|\hat{x}(\varphi_{\epsilon})+1|\nonumber\\
&\ge&1/2.
\end{eqnarray}
Now, if $\varphi\notin T$, then by definition of $T$, we have that
$|\hat{x}(\varphi_\varepsilon)|>\frac{1}{2}$ and $\hat{e}_T(\varphi_\varepsilon)=0$, therefore 
\begin{eqnarray}\label{exchan-2}
|(\hat{x}+\hat{e}_T)(\varphi_\varepsilon)|&=&|\hat{x}(\varphi_\varepsilon)+\hat{e}_T(\varphi_\varepsilon)|\nonumber\\
&=&|\hat{x}(\varphi_\varepsilon)|\nonumber\\
&>&\frac{1}{2}
\end{eqnarray}
and we consider $r$ as above. Hence, in both cases, we have from
(\ref{exchan-1}) and (\ref{exchan-2}) that $x+e_T$ is invertible element, 
and we have that $\overline{\mathbb{R}}$ is an exchange ring.
\end{proof}

Note that the Proposition \ref{colom-17} holds for $\overline{\mathbb{K}}, i.e.,
\mathbb{K}=\mathbb{C}$ or $\mathbb{K}=\mathbb{R}$. In the next result, we use Proposition \ref{colom-17} for
$\mathbb{K}=\mathbb{R}$ to prove it.

\begin{teo}\label{colom-18}
$\overline{\mathbb{H}}$ is an exchange ring. 
\end{teo}

\begin{proof}  Let $(a+a_1i+a_2j+a_3k)\in \overline{\mathbb{H}}$. 
Since $\overline{\mathbb{R}}$ is an exchange ring, there exists an 
idempotent $e\in \overline{\mathbb{R}}$, such that $a_0+e$ is
invertible, So, by Lemma \ref{unit} $(a+e)=(a_0+e)+a_1i+a_2j+a_3k$ 
is invertible in $\overline{\mathbb{H}}$ and it follows the result.
\end{proof}

The next definition is well-known and we can see an example 
in \cite{MAT,LT} and \cite{MJ}

\begin{defn}\label{colom-19}
Let $S$ be a noncommutative reduced ring with an order
``$\leq$''. We say that $S$ is nornal if for two minimal prime ideals 
$\mathfrak{P}_1$, $\mathfrak{P}_2$ of $S$ such that
$\mathfrak{P_1}\neq\mathfrak{P}_2$, we have
$S=\mathfrak{P}_1+\mathfrak{P}_2$. 
\end{defn}

\begin{teo}\label{colom-20}
$\overline{\mathbb{H}}$ is normal  
\end{teo}

\begin{proof} Let $\mathfrak{P}_1$, $\mathfrak{P}_2$ be minimal prime ideals of
  $\overline{\mathbb{H}}$. Then by Theorem \ref{ideal-1}
  $\mathfrak{P}_1=\overline{\mathbb{H}}(g_f(\mathcal{F}_1))$ 
and $\mathfrak{P}_2=\overline{\mathbb{H}}(g_f(\mathcal{F}_2))$. 
By similar methods of (\cite{VH}, Corollary 2.4) we have that 
$\overline{\mathbb{R}}$ is a commutative normal ring and we get that 
$\overline{\mathbb{R}}=g_f(\mathcal{F}_1)+g_f(\mathcal{F}_2)$. 
Thus,
$\overline{\mathbb{H}}=\overline{\mathbb{H}}(g_f(\mathcal{F}_1+\mathcal{F}_2))=
\overline{\mathbb{H}}(g_f(\mathcal{F}_1))+\overline{\mathbb{H}}(g_f(\mathcal{F}_2))$. So,
$\overline{\mathbb{H}}$ is normal. 
\end{proof}

The next definition is well-know, and we can see, for example, in
\cite{SHS-1,SHS-2} and \cite{BF}.

\begin{defn}\label{colom-21} A not necessarily commutative ring $S$ is
right Gelfand if for any right maximal ideals $\mathfrak{M}$ and $\mathfrak{N}$ of $S$ there exists 
$a,b\in S$ such that $a\notin \mathfrak{M}$, $b\notin \mathfrak{N}$ with $aRb=0$. 
Similarly we have left Gelfand rings. Moreover, a ring is 
Gelfand if it is right and left Gelfand. 
\end{defn}

\begin{teo}\label{colom-22}
$\overline{\mathbb{H}}$ is Gelfand.
\end{teo}

\begin{proof} By Theorem \ref{colom-14} we have that all 
right and left ideals are ideals. Thus,  
in this case all right and left maximal ideals are maximal ideals. 
Let $\mathfrak{M}$ and $\mathfrak{N}$  maximal ideals of
$\overline{\mathbb{H}}$. Then by 
Theorem \ref{ideal-1} we have that $\mathfrak{M}=\overline{\mathbb{H}}(\overline{g(\mathcal{F}_1)})$ 
and $\mathfrak{N}=\overline{\mathbb{H}}(\overline{g(\mathcal{F}_2)})$. By similar methods
of (\cite{VH}, Section 2) we have that $\overline{\mathbb{R}}$ 
is Gelfand and we obtain that there exists $a,b\in
\overline{\mathbb{R}}, ~a\notin \overline{g(\mathcal{F}_1)}$ and 
$b\notin \overline{g(\mathcal{F}_2)}$, such that $a\overline{\mathbb{R}}b=0$.  
Hence, $a\notin \mathfrak{M}$ and $b\notin \mathfrak{N}$ such that 
$a\overline{\mathbb{H}}b=0$. Therefore, $\overline{\mathbb{H}}$ is Gelfand.
\end{proof}

According to \cite{WJ} a not necessary commutive ring $S$ 
is right Bezout if all finitely generated ideals are principal. 
Analogously, we define left Bezout. Moreover, a ring is Bezout 
if it is right and left Bezout.

\begin{teo}\label{colom-23}
$\overline{\mathbb{H}}$ is Bezout 
\end{teo}

\begin{proof} We claim that for any $a\in \overline{\mathbb{H}}$, 
we have that $$\langle n(a)\overline{\mathbb{R}}\rangle =
\{x\in\overline{\mathbb{H}}:n(x)\in n(a)\overline{\mathbb{R}}\}=n(a)\overline{\mathbb{H}}.$$ 
In fact, for each $y\in a\overline{\mathbb{H}}$, we get that 
$n(y)\in a\overline{\mathbb{H}}\cap \overline{\mathbb{R}}$. 
Thus, $n(y)=ax$ which implies that $n(y)=n(ax)=n(a)n(x)\in
n(a)\overline{\mathbb{R}}$ and we have that  
$y\in \langle n(a)\overline{\mathbb{R}}\rangle$. 
On the other hand, let $z\in  \langle
n(a)\overline{\mathbb{R}}\rangle$. 
Then, $n(z)\in
n(a)\overline{\mathbb{R}}=a\Theta_{\hat{a}}\overline{\mathbb{R}}\in a
\overline{\mathbb{H}}\cap \overline{\mathbb{R}}$. Hence, $n(z)\in
a\overline{\mathbb{H}}$ 
and by Proposition \ref{colom-13}, $z\in a\overline{\mathbb{H}}$. 
Finally, for each $a,b\in \overline{\mathbb{H}}$ we have that
\begin{eqnarray*} 
a\overline{\mathbb{H}}+b\overline{\mathbb{H}}&=&\langle
n(a)\overline{\mathbb{R}}\rangle + 
\langle n(b) \overline{\mathbb{R}}\rangle\\
&=&\langle n(a)\overline{\mathbb{R}}+n(b)
\overline{\mathbb{R}}\rangle\\
&=&\langle (n(a)+n(b))\overline{\mathbb{R}}\rangle\\
&=&(n(a)+n(b))\overline{\mathbb{H}}
\end{eqnarray*}
which proves our assertion.
\end{proof}

In the next result, we show that all the ideals of
$\overline{\mathbb{H}}$ are idempotent. 

\begin{prop}\label{colom-24}
All the ideals of $\overline{\mathbb{H}}$ are idempotent.
\end{prop}

\begin{proof}  
Let $\mathfrak{I}$ be an ideal of $\overline{\mathbb{H}}$. 
Then we easily have that $\mathfrak{I}^2\subseteq\mathfrak{I}$. 
On the other, let $a\in\mathfrak{I}$. Then  by Proposition
\ref{colom-13} we have that $n(a)\in \mathfrak{I}$. Thus,
$n(a)\leq n(a)^2\in \mathfrak{I}^2$ 
and again by Proposition \ref{colom-13}, $n(a)\in \mathfrak{I}^2$ which implies  by 
Proposition \ref{colom-13}  that $a\in \mathfrak{I}^2$. Hence, 
$\mathfrak{I}\subseteq \mathfrak{I}^2$. So, $\mathfrak{I}=\mathfrak{I}^2$. 
\end{proof}

\begin{obs}\label{colom-25} Using the same techiniques as above, 
we get that all ideals of $\overline{\mathbb{R}}$ are idempotent.
\end{obs}

Now, we introduced the concept of pseudo-prime rings for noncommutative rings.

\begin{defn}\label{colom-26}  Let $S$ be a not necessarily commutative ring and $\mathfrak{I}$ an
  ideal of $S$. We say that $\mathfrak{I}$ is pseudo-prime if for
  ideals $\mathfrak{J}$ and $\mathfrak{K}$ of $S$ such
that $\mathfrak{J}\mathfrak{K}=0$ implies that
$\mathfrak{J}\subseteq\mathfrak{I}$ and $\mathfrak{K}\subseteq\mathfrak{I}$.
\end{defn}

In Definition \ref{colom-26}, if $S$ is commutative, we have the
definition presented in \cite{VH}.

Using the similar methods  of (\cite{VH}, Theorem 4.6) and Remark \ref{colom-25},
we have the following result.

\begin{lema}\label{colom-27} 
Let $\mathfrak{I}$ be an ideal of $\overline{\mathbb{R}}$. Then $\mathfrak{I}$
is pseudo-prime if and only if $\mathfrak{I}$ is prime. 
\end{lema}

Next, we show that all pseudo-prime ideals of 
$\overline{\mathbb{H}}$ are in fact prime ideals.

\begin{prop}\label{colom-28} 
Let $\mathfrak{I}$ be an ideal of $\overline{\mathbb{H}}$. Then
$\mathfrak{I}$ is prime if and only if it is pseudo-prime.
\end{prop}

\begin{proof}  Suppose that $\mathfrak{I}$ is pseudo prime. 
We claim that $\mathfrak{I}\cap
  \overline{\mathbb{R}}$ is pseudo-prime. In fact, let $a,b\in
  \overline{\mathbb{R}}$ such that $ab=0$. Then 
$a\overline{\mathbb{H}}b\overline{\mathbb{H}}=ab\overline{\mathbb{H}}=0$. 
By assumption, we have that $a\overline{\mathbb{H}}\subseteq \mathfrak{I}$ or
$b\overline{\mathbb{H}}\subseteq \mathfrak{I}$ 
and it follows that $a\in \mathfrak{I}\cap \overline{R}$ or 
$b\in \mathfrak{I}\cap  \overline{\mathbb{R}}$. 
Thus, by Lemma \ref{colom-27} $\mathfrak{I}\cap \overline{\mathbb{R}}$ is prime.

Next, we show that $\mathfrak{I}$ is prime. 
In fact, let $\mathfrak{J}$ and $\mathfrak{K}$ be ideals of
$\overline{\mathbb{H}}$ such that $\mathfrak{J}\mathfrak{K}\subseteq
\mathfrak{I}$. Then
$(\mathfrak{J}\cap\overline{\mathbb{R}})(\mathfrak{K}\cap
\overline{\mathbb{R}}) \subseteq \mathfrak{I}\cap
\overline{\mathbb{R}}$. Thus, we have that 
$(\mathfrak{J}\cap \overline{\mathbb{R}})\subseteq \mathfrak{I}$ 
or $(\mathfrak{K}\cap \overline{\mathbb{R}})\subseteq   \mathfrak{I}$.  
So, $\langle \mathfrak{J}\cap \overline{\mathbb{R}} \rangle=\{a\in
\overline{\mathbb{H}}:n(a)\in \mathfrak{J}\cap
\overline{\mathbb{R}}\}=\mathfrak{J}\subseteq \mathfrak{I}$ 
or $\langle \mathfrak{K}\cap \overline{\mathbb{R}}\rangle=\{a\in
\overline{\mathbb{H}}: n(a)\in \mathfrak{K}\cap
\overline{\mathbb{R}}\} 
=\mathfrak{K}\subseteq \mathfrak{I}$.
Hence, $\mathfrak{I}$ is prime.
\end{proof}

The next  definitions appears in (\cite{LAM-2}, Definitions 4.10.3 and 4.10.6).

\begin{defn}\label{colom-29}  
\begin{enumerate}
\item[$i)$] Let $S$ be a ring. A nonempty set $M\subseteq S$ is
  called an $m$-system if, 
for any $a,b\in M$ there exists $s\in S$ such that $asb\in M$. 
\item[$ii)$] Let $S$ be a ring. For an ideal $\mathfrak{U}$ of $S$, the radical of $\mathfrak{U}$
is  $$\sqrt{\mathfrak{U}}=\{s\in S:\mbox{every} ~m-\mbox{system}
~\mbox{containing} ~s ~\mbox{meets} ~\mathfrak{U}\}.$$
\end{enumerate} 
\end {defn}

An ideal $\mathfrak{U}$ of $\overline{\mathbb{H}}$ is radical of $\mathfrak{U}=\sqrt{\mathfrak{U}}$. 
In the next result, we show that all the ideals of $\overline{\mathbb{H}}$ 
and $\overline{\mathbb{R}}$ are radicals.

\begin{prop}\label{colom-30} 
\begin{enumerate}
\item[$i)$] All the ideals of $\overline{\mathbb{H}}$ are radicals;
 \item[$ii)$] All the ideals of $\overline{\mathbb{R}}$ are radicals. 
\end{enumerate}
\end{prop}
 
\begin{proof} $i)$: Let $\mathfrak{I}$ be an ideal of $\overline{\mathbb{H}}$ and
   $x\in \sqrt{\mathfrak{I}}$. 
Then there exists $n\geq 1$ such that $a^n\in \mathfrak{I}$, since the set
$\{1,a^2,...,a^n..\}$ is a $m$-system. Then by Proposition \ref{colom-13}
$n(a)^n\in \mathfrak{I}$ and note that $n(a)\leq n(a)^n\in \mathfrak{I}$. 
Hence, by Proposition \ref{colom-13} we have that $n(a)\in \mathfrak{I}$.
So, $a\in \mathfrak{I}$ and it follows that $\mathfrak{I}=\sqrt{\mathfrak{I}}$.
 
$ii)$: The proof follows the similar methods of $i)$. 
\end{proof}
 
In the next results, we characterize the essential ideals\footnote{A
  closed ideal $\mathfrak{I}$ in a $C^*$-algebra $A$ is called essential 
if $\mathfrak{I}$ has nonzero intersection with every other nonzero closed ideal $A$
or, equivalently, if $a\mathfrak{I}=\{0\}$ implies $a=0$ for all $a\in
A$ (Raeburn and Williams 1998).

In mathematics, specifically module theory, given a ring $R$ and
$R$-modules $M$ with a submodule $N$, the module $M$ is said to be 
an essential extension of $N$ (or $N$ is said to be an essential
submodule or large submodule of $M$) if for every submodule $H$ of $M$,
$H\cap N=\{0\}$ implies that $H=\{0\}$. As a special case, an
essential left ideal of $R$ is a left ideal which is essential as a 
submodule of the left module $R_R$. The left ideal has non-zero
intersection with any non-zero left ideal of $R$. Analogously, the
essential right ideal is exactly an essential submodule of the right $R$ module $R_R$} of
$\overline{\mathbb{R}}$ and $\overline{\mathbb{H}}$. We shall use
the notation of \cite{BMW}.

\begin{lema}\label{ideal-5}
Let $\mathfrak{I}$ be an ideal of $\overline{\mathbb{R}}$. Then
$r_{\overline{\mathbb{R}}}(\mathfrak{I})\ne \{0\}$ if and only if
there exists an idempotent $e\in\mathcal{B}(\overline{\mathbb{R}})$,
such that $\mathfrak{I}\subseteq\overline{\mathbb{R}}
e$. Equivalently, $\mathfrak{I}$ is essential if and only if it is
not contained in a principal idempotent ideal. Moreover, if
$r_{\overline{\mathbb{R}}}(\mathfrak{I})\ne \{0\}$, then
$\mathcal{B}(\overline{\mathbb{R}})\cap
r_{\overline{\mathbb{R}}}(\mathfrak{I})\ne \{0\}$.  
\end{lema}

\begin{proof}
$(\Leftarrow)$: If there exists an idempotent
$e\in\mathcal{B}(\overline{\mathbb{R}})$, such that
$\mathfrak{I}\subseteq \overline{\mathbb{R}}e$, then $(1-e)\in
r_{\overline{\mathbb{R}}}(\mathfrak{I})$, and we get ,
$r_{\overline{\mathbb{R}}}(\mathfrak{I})\ne \{0\}$.

$(\Rightarrow)$: Conversely, if $0\neq x\in
r_{\overline{\mathbb{R}}}(J)$, then $x$ must be a zero divisor and it
  follows that there exists $A\in\mathcal{S}_f$, such that
  $x\mathcal{X}_A=0$ or equivalently $x\mathcal{X}_{A^c}=x$. We claim
  that $\mathcal{X}_{A^c}\in
  r_{\overline{\mathbb{R}}}(\mathfrak{I})$. Indeed, for any
  $y\in\mathfrak{I}$, we have that $xy=0\Rightarrow
  (x\mathcal{X}_{A^c})y=0\Rightarrow x\mathcal{X}_{A^c}y=0 $. So, if
  we choose representatives, we have that
  $\hat{x}(1-\mathcal{X}_A)\hat{y}\in\mathcal{N}(\mathbb{R})$. We
  will show that
  $(1-\mathcal{X}_A)\hat{y}=\mathcal{X}_{A^c}\hat{y}\in\mathcal{N}_f(\mathbb{R})$. In
  fact, if
  $\mathcal{X}_{A^c}\hat{y}\notin\mathcal{N}_f(\mathbb{R})$. then
  $V(\mathcal{X}_{A^c}\hat{y})<\infty$, i.e., there exists
  $a\in\mathbb{R}^+$, such that
  $V(\mathcal{X}_{A^c}\hat{y})=\sup(A(\mathcal{X}_{A^c}\hat{y}))=a$,
  and consequently there exists $p\in\mathbb{N}$, such
  that $$\lim\limits_{\varepsilon\downarrow
    0}\frac{\mathcal{X}_{A^c}(\varphi_\varepsilon)\hat{y}(\varphi_\varepsilon)}{\varepsilon^a}\ne
  0.$$ However, $$0=\lim\limits_{\varepsilon\downarrow
    0}\frac{\hat{x}(\varphi_\varepsilon)\mathcal{X}_{A^c}(\varphi_\varepsilon)\hat{y}(\varphi_\varepsilon)}
{\varepsilon^{a+b}}=\lim\limits_{\varepsilon\downarrow
    0}\frac{\hat{x}(\varphi_\varepsilon)\mathcal{X}_{A^c}(\varphi_\varepsilon)}{\varepsilon^b}\lim\limits_{\varepsilon\downarrow
    0}\frac{\mathcal{X}_{A^c}(\varphi_\varepsilon)\hat{y}(\varphi_\varepsilon)}{\varepsilon^a},$$
  which implies that $$\lim\limits_{\varepsilon\downarrow
    0}\frac{\hat{x}(\varphi_\varepsilon)\mathcal{X}_{A^c}(\varphi_\varepsilon)}{\varepsilon^b}=0,
  ~\forall~b\in\mathbb{R},$$
and it follows that
  $A(\hat{x}\mathcal{X}_{A^c})=\mathbb{R}\Rightarrow
  V(\hat{x}\mathcal{X}_{A^c})=+\infty\Rightarrow
  \hat{x}\mathcal{X}_{A^c}\in\mathcal{N}_f(\mathbb{R}).$ So,
  $x\mathcal{X}_{A^c}=0$,  we have that $x=x\mathcal{X}_A+x\mathcal{X}_{A^c}=0$,
  which is a contradiction. Therefore,
  $\mathcal{X}_{A^c}\hat{y}\in\mathcal{N}_f(\mathbb{R})$ and it follows
  that $\mathcal{X}_{A^c}y=0, ~\forall~y\in\mathfrak{I}$, and we have that
  $\mathcal{X}_{A^c}\in r_{\overline{\mathbb{R}}}(\mathfrak{I})$. Using this, we have
  for $y\in\mathfrak{I}$ that
  $y=y\mathcal{X}_A+y\mathcal{X}_{A^c}=y\mathcal{X}_A\in
\overline{\mathbb{R}}\mathcal{X}_A$.
\end{proof}

Lemma \ref{ideal-6}, Lemma \ref{ideal-7} and Lemma \ref{ideal-8} below, which
extends the results of (\cite{CFJ}, Lemma 4.7, Lemma 4.8 and Lemma 4.9)
have their proves with similar techniques of \cite{CFJ}, and their
proves will be omitted here.

\begin{lema}\label{ideal-6}
Any  proper finitely generated ideal of $\overline{\mathbb{K}}$
 is contained in a principal idempontent ideal. Hence, it is
not essential. In particular, essential ideals are not finitely generated.  
\end{lema}

\begin{lema}\label{ideal-7}
Let $\mathfrak{I}$ be an ideal of $\overline{\mathbb{H}}$. Then
$\mathfrak{I}\in\mathcal{D}(\overline{\mathbb{H}})$ iff 
$n(\mathfrak{I})\in\mathcal{D}(\overline{\mathbb{H}})$, where
$\mathcal{D}(\overline{\mathbb{H}})$ 
is the set of all essential ideals of $\mathbb{\overline{H}}$. 
\end{lema}

\begin{lema}\label{ideal-8}
Let $\mathfrak{J}\lhd\overline{\mathbb{H}}$ be an ideal. Then
$\mathfrak{I}\subset\overline{\mathbb{H}}(n(\mathfrak{I}))$. Moreover,
if $\mathfrak{I}$ is semi-prime, then $\mathfrak{I}=\overline{\mathbb{H}}(n(\mathfrak{I}))$.
\end{lema}

Proposition \ref{ideal-9} which extends the result of
(\cite{CFJ}, Proposition 4.10) has its prove as 
that appears in \cite{CFJ} and therefore will 
be omitted. 

\begin{prop}\label{ideal-9}
Let $\mathfrak{I}\lhd\overline{\mathbb{H}}$ be an ideal. The
following conditions are equivalent:
\begin{enumerate}
\item[$a)$] $\mathfrak{I}\notin\mathcal{D}(\overline{\mathbb{H}})$;
\item[$b)$] There exists an idempotent $e\in\overline{\mathbb{R}}$, 
such that $\mathfrak{I}\subseteq\overline{\mathbb{H}}e$.
\end{enumerate}
\end{prop}

Here we let $\mathbb{K}$ stand for $\mathbb{R}$ or $\mathbb{C}$. It
follows from Proposition \ref{ideal-9} that the singular ideals
$Z_r(\overline{\mathbb{K}})=Z_r(\overline{\mathbb{H}})=\{0\}$ and
hence, by Theorem 2.1.15 of \cite{BMW}, we have that
$Q_{\max}(\overline{\mathbb{K}})$ and
$Q_{mr}(\overline{\mathbb{H}})$, the maximal right ring of
quotients of $\overline{\mathbb{K}}$ are Von Neumann
regular. Moreover, it is not difficult to see that  
$Q_{mr}(\overline{\mathbb{K}})$ is contained $Q_{mr}(\overline{\mathbb{H}})$. 

\vspace{1.5cm}
{\bf{Aknowledgement}}: The authors are grateful to UAMat/CCT/UFCG and CCEN/UFERSA
where this work was realized. The second author thanks PNPD/CAPES for
their support.

\bibliographystyle{plain}
\bibliography{biblio_gf}

\Addresses

\end{document}